\documentclass[12pt, reqno]{amsart}
\usepackage{amssymb}
\usepackage{graphicx}
\usepackage{xcolor} 
\usepackage{tensor}
\usepackage[color=yellow]{todonotes}

\usepackage{enumitem}
\usepackage{fullpage} 
\usepackage{hyperref}

\definecolor{green}{rgb}{0,0.8,0} 

\newtheorem{theorem}{Theorem}[section]
\newtheorem{corollary}[theorem]{Corollary}
\newtheorem{lemma}[theorem]{Lemma}

\theoremstyle{definition}
\newtheorem{definition}[theorem]{Definition}

\theoremstyle{remark}
\newtheorem{remark}[theorem]{Remark}
\numberwithin{equation}{section}

\newcommand{\bbR}{\mathbb R}

\setlist[enumerate]{leftmargin=2em, label=(\arabic*)}
\setlist[itemize]{leftmargin=2em}

\title[]{Well-posedness and regularity for the fractional Navier-Stokes equations}
\author{Ning Tang}%
\address{Zhejiang University, Hangzhou, China}%
\email{math\_tn@berkeley.edu 3180102677@zju.edu.cn}%

\date{today}%

\begin{document}
\maketitle
\setcounter{tocdepth}{1}
\tableofcontents

\section{Introduction}
Motivated by the classical paper \cite{ref0} by Koch and Tataru,
we study the generalized Navier-Stokes equation in $\bbR^n$
\begin{equation}
    \begin{cases}
        u_t + (u\cdot\nabla)u+ (-\Delta)^\alpha u + \nabla p = 0\\
        \nabla\cdot u = 0\\
        u(x,0) = u_0(x),\\
    \end{cases}     
\end{equation}
where $\alpha\in(\frac12,1)$ is a real parameter.
In \cite{ref0}, they study the case $\alpha = 1$.
For the case $\alpha\in(\frac12,1)$, heuristically speaking,
the effect of the dissipative term is still stronger than the transport term. 

Note that $(-\Delta )^\alpha$ is the Fourier multiplier defined by the obvious way
\[
  \widehat{(-\Delta )^\alpha u}(\xi) = |\xi|^{2\alpha}\hat u(\xi).
\]
We use the Leray projection to rewrite the equation as 
\[
  u_t + (-\Delta)^\alpha u  = -\nabla\Pi Nu,  \quad N(u) = u\otimes u.
\] 

Now we denote the fractional heat semigroup by $S(t) = e^{-t(-\Delta)^\alpha}$ 
without specifying the domain,
where 
$$S(s)f(t) = \mathcal{F}^{-1}\left(\int e^{-s|\xi|^{2\alpha}}\hat f(\xi, t)\, d\xi\right).$$
Let $\Phi(x) = \mathcal{F}^{-1}\left(e^{-|\xi|^{2\alpha}}\right)$,
and $\Phi_t(x) = \frac1{t^n} \Phi(\frac{x}{t})$.
Then the kernel of the fractional heat semigroup $S(t)$ 
can be denoted by 
$$\Phi_{t^{\frac1{2\alpha}}}(x) = (t^{\frac1{2\alpha}})^{-n} \Phi(t^{-\frac1{2\alpha}} x) := \Phi(t,x).$$
And the kernel satisfies the estimate below    
$$|\Phi(x,t)|\lesssim \frac{t}{(t^{1/\alpha}+|x|^2)^{(n+2\alpha)/2}},$$
which can be found in \cite{ref1}.
In particular, $\Phi(t,x)$ is integrable.

From now on, we only consider the case $\frac12 < \alpha < 1$ unless other specified.
Let now $Q(x, R) = B(0,R)\times (0,R^{2\alpha})$.
\begin{definition}\label{definiton_BMO_-1}
    Now we define the space of tempered distributions $\mathcal{S}'(\mathbb R^n)$ with the following norm 
    $$
    \|v\|_{\mathcal{E}} = \sup_{x, R>0} \left(R^{-(n+2-2\alpha)}\int_{Q(x,R)} |S(t)v(y)|^2\, dt\, dy\right)^{\frac12}
    $$
    is finite. And the space is denoted by $\mathcal{E}$.
\end{definition}
The above definition is motivated by some scaling invariance property and inspired from \cite{ref0}.
Furthermore, when $\alpha\in (\frac12, 1)$, 
we could expect more properties than the case $\alpha = 1$.
We use the fractional heat kernel to present an equivalent definition of Besov space:
\begin{lemma}\label{Besov_norm_equivalence}
   The norm of the homogeneous Besov space $\dot B^{1-2\alpha,\infty}_\infty$ is
   $$\|v\|_{\dot B^{1-2\alpha,\infty}_\infty} \sim \sup_{t>0} t^{1-\frac1{2\alpha}}\|S(t)v(x)\|_{L^\infty(\mathbb R^n)}$$
   when $\alpha\in (\frac12, 1)$.
\end{lemma}
Note that the equivalence to the usual the definition of Besov space can be proved
following an analogous proof of \cite[Theorem 5.3]{ref2}. We defer the proof to next section.

The above discussion motivates the introduction of the spaces $GX, GY$
of functions in $\mathbb R^n\times \mathbb R_+$ with norms
$$\|u\|_{GX} = \sup_{t>0} t^{1-\frac1{2\alpha}}\|u(t)\|_{L^\infty(\mathbb R^n)},$$
$$\|f\|_{GY} = \sup_{t>0} t^{2-\frac1{\alpha}}\|f(t)\|_{L^\infty(\mathbb R^n)}.$$

Actually, some of these function spaces are equivalent.
\begin{theorem}\label{equivalence_of_norms}
    For $h\in\mathcal{S}'(\mathbb R^n)$, the following statements are equivalent:
    \begin{enumerate}
        \item $h(x)\in \dot B^{1/\alpha-2,\infty}_\infty$.
        \item $h(x)\in \mathcal{E}$.
        \item $S(t)h(x)\in GX$.
    \end{enumerate}
    Moreover, $$\|h(x)\|_{\dot B^{1/\alpha-2,\infty}_\infty} \sim \|h(x)\|_{\mathcal{E}}\sim \|S(t)h(x)\|_{GX}.$$
    And for simplicity, 
    we denote $\dot B^{1/\alpha-2,\infty}_\infty$, $\mathcal{E}$ by $\mathcal{B}$.
\end{theorem}
We leave the proof to the following section.

\bigskip
Now we could state our main result for the wellposedness of generalized Navier-Stokes:
\begin{theorem}\label{main_theorem}
    The generalized Navier-Stokes equations 
    \begin{equation}
        \begin{cases}
            u_t + (u\cdot\nabla)u +(-\Delta)^\alpha u + \nabla p = 0\\
            \nabla\cdot u = 0\\
            u(0) = u_0\\
        \end{cases}     
    \end{equation}
    have a unique small global solution in $\mathcal{B}$ for all initial data $u_0$ with $\nabla \cdot u_0 = 0$ which are small in $GX$,
    where $\alpha\in(\frac12, 1)$.
\end{theorem}

\begin{corollary}\label{corollary_main_theorem}
    The global solution constructed in Theorem~\ref{main_theorem} satisfies that 
    $$u(x, t_0) \in \mathcal{B}$$
    for each fixed $t_0 > 0$.
\end{corollary}

Note that this corollary is not true for $\alpha = 1$.
In fact, the solution map is not continuous in $\dot B^{-1,\infty}_\infty$.
See \cite{ref6} for further discussion.

\subsection*{Acknowledgments} 
This note is written based on an undergraduate summer research project 
supervised by Professor Sung-Jin Oh at Berkeley in 2021. 
The author would like to thank him for introducing this topic and a lot of helpful discussions. 
Moreover, specially thanks for his generous help remotely during the pandemic.

\section{Proof of Theorem~\ref{equivalence_of_norms}} 
Thanks to scaling invariance, we present an important lemma at first.
\begin{lemma}
    $$ |S(t)u_0(x)| \leq Ct^{-\left(1-\frac1{2\alpha}\right)}\|u_0\|_{\mathcal{E}}$$
\end{lemma}
\begin{proof}
    First, we show some scaling property.
    Note that $\mathcal{F}(u_0(\lambda x)) = \lambda^{-n}\widehat{u_0}(\frac{\xi}{\lambda})$,
    then 
    \begin{equation}\label{scaling_fractional_temperature}
        \begin{aligned}
            (S(t)u_0)(x)
         =& \int e^{ix\cdot\xi}e^{-t|\xi|^{2\alpha}}\widehat{u_0}(\xi)\, d\xi
         = \left(\lambda^{\frac1{2\alpha}}\right)^{n} 
         \int e^{i\lambda^{\frac1{2\alpha}}x\cdot\xi }e^{-t\lambda|\xi|^{2\alpha}}\widehat{u_0}(\lambda^{\frac1{2\alpha}}\xi)\, d\xi\\
         =& \int e^{i\lambda^{\frac1{2\alpha}}x\cdot\xi }e^{-t\lambda|\xi|^{2\alpha}}\widehat{v_0}(\xi)\, d\xi
         = (S(t\lambda)v_0)(\lambda^{\frac1{2\alpha}}x) 
         ,\end{aligned}
    \end{equation}
    where $v_0(y) = u_0(y/\lambda^{\frac1{2\alpha}})$ and $\lambda > 0$.
    And from \eqref{scaling_fractional_temperature}, we have
    \begin{equation}\label{scaling_L_2_average}
        \begin{aligned}
            \left(R^{-(n+2-2\alpha)}\int_{Q(x,R)} |S(t)u_0(y)|^2\, dt\, dy\right)^{\frac12}
            = \left(R^{-(n+2-2\alpha)}\int_0^{R^{2\alpha}}\int_{B(x,R)} |S(\lambda t)v_0(\lambda^{\frac1{2\alpha}}y)|^2\, dy\, dt\right)^{\frac12}&\\
            = \lambda^{-1+\frac1{2\alpha}}\left((\lambda^{-\frac1{2\alpha}}R)^{-(n+2-2\alpha)}\int_{Q(x,\lambda^{-\frac1{2\alpha}}R)} |S(t)v_0(y)|^2\, dt\, dy\right)^{\frac12}&    
        ,\end{aligned}
    \end{equation}
    where the last equality follows from a change of variable. 
    Then taking supremum implies
    $$\|u_0\|_{\mathcal{E}} = \lambda^{-\left(1-\frac1{2\alpha}\right)} \|v_0\|_{\mathcal{E}}.$$
    And note that $\|u_0\|_{\mathcal{E}}$ is translational invariant, 
    thus it suffices to show 
    $$|S(1)u_0(0)| \leq C\|u_0\|_{\mathcal{E}}.$$

    Since 
    $$|S(1)u_0(0)|^2 = \left|\int_{\mathbb R^n}\Phi(1-s, -y) (S(s)u_0)(y)\, dy\right|^2
    \lesssim \int_{\mathbb R^n}\Phi(1-s, -y) |(S(s)u_0)(y)|^2\, dy
    ,$$
    then replace the right hand side by its time average, the inequality still holds that 
    $$
    \begin{aligned}
      |S(1)u_0(0)|^2 \lesssim& 2^{2\alpha}\int_0^{2^{-2\alpha}}\int_{\mathbb R^n}\Phi(1-s, -y) |(S(s)u_0)(y)|^2\, dy\, ds\\
    \lesssim& \sum_{k\in\mathbb Z^n} \int_0^{2^{-2\alpha}}\int_{B(\frac{1}{\sqrt{2n}}k, \frac12)} 
    \Phi(1-s, -y) |(S(s)u_0)(y)|^2 \, dy\, ds\\
    \lesssim& \|u_0\|_{\mathcal{E}} \left(\sum_{k\in\mathbb Z^n} \sup_{s\in[0, 2^{-2\alpha}] \atop y\in B(\frac{1}{\sqrt{2n}}k, \frac12)} |\Phi(1-s, -y)| \right) 
    \lesssim \|u_0\|_{\mathcal{E}}\\  
    \end{aligned}
    ,$$
    where the last inequality comes from the kernel estimate 
    $$|\Phi(x,t)|\lesssim \frac{t}{(t^{1/\alpha}+|x|^2)^{(n+2\alpha)/2}}$$
    which can be found in \cite{ref1}.
    Then the result follows.
\end{proof}

Furthermore, it is straightforward that 
$$
\begin{aligned}\label{reverse_estimate}
    &\left(R^{-(n+2-2\alpha)}\int_{Q(x,R)} |S(t)u_0(y)|^2\, dt\, dy\right)^{\frac12}\\
    \leq& \left(\frac1{R^{2-2\alpha}}\int_0^{R^{2\alpha}} t^{-\left(2-\frac1{\alpha}\right)}\, dt\right)  \sup_{t>0} t^{1-\frac1{2\alpha}}\|S(t)u_0(y)\|_{L^\infty}
    \lesssim_\alpha \sup_{t>0} t^{\left(1-\frac1{2\alpha}\right)}\|S(t)u_0(y)\|_{L^\infty}      
,\end{aligned}
$$
where the last equality follows from the assumption that $\alpha\in (\frac12, 1)$.
Thus we know 
\begin{equation}\label{equivalence_norm_E}
    \|v\|_{\mathcal{E}} \sim \sup_{t>0}t^{1-\frac1{2\alpha}}\|S(t)v(x)\|_{L^\infty(\mathbb R^n)}
\end{equation}
are equivalent norms.

\bigskip
Now we prove Lemma~\ref{Besov_norm_equivalence}.
We state a more precise version of the lemma:
Let $\sigma < 0$, $f\in\mathcal{S}_h'(\mathbb R^n)$, $S(t) = e^{-t(-\Delta)^\alpha}$ is the fractional heat kernel,
(the definition of $\mathcal{S}_h'$ can be found in \cite[Definiton 1.26]{ref3}),
then we have 
\begin{equation}\label{precise_version_of_lemma}
    \sup_{t>0}t^{-\frac{\sigma}{2\alpha}}\|S(t)f\|_{L^\infty} \lesssim \|f\|_{\dot B^{\sigma,\infty}_\infty} \lesssim \sup_{t>0}t^{-\frac{\sigma}{2\alpha}}\|S(t)f\|_{L^\infty}.
\end{equation}

\begin{proof}
    \textbf{The first inequality}:
    By the Littlewood Paley decomposition characterization of Besov spaces, (c.f. \cite{ref3})
    $f = \sum_{j\in\mathbb Z} P_j f$ is true for $f\in\mathcal{S}_h'$, 
    with $P_j f\in L^\infty(\mathbb R^n)$ for $j\in\mathbb Z$ 
    and $\|P_j f\|_{L^\infty(\mathbb R^n)} = 2^{-j\sigma}\epsilon_j$ where $(\epsilon_j)_{j\in\mathbb Z}\in \ell^\infty$.
    We estimate the norm $t^{-\frac{\sigma}{2\alpha}}\|S(t)f\|_{L^\infty}$ by
    $$t^{-\frac{\sigma}{2\alpha}}\|S(t)P_j f\|_{L^\infty}
    \leq t^{-\frac{\sigma}{2\alpha}}\|P_j f\|_{L^\infty}.$$
    Furthermore, we write 
    $$S(t)P_j f = t^{-\frac{N}{2\alpha}} \left(t(-\Delta)^{\alpha}\right)^{\frac{N}{2\alpha}} 
    S(t) (-\Delta)^{-\frac{N}2} \widetilde{P_j} P_j f,$$
    then since the kernel of $\left(t(-\Delta)^{\alpha}\right)^{\frac{N}{2\alpha}} S(t)$ (Lemma~\ref{integrability_fractional_temperature}) and 
    $(-\Delta)^{-\frac{N}2} \widetilde{P_0} = 2^{jN}(-\Delta)^{-\frac{N}2}\widetilde{P_j}$ are integrable, respectively,
    therefore we have 
    $$
    \begin{aligned}
        &t^{-\frac{\sigma}{2\alpha}}\|S(t)P_j f\|_{L^\infty}
        \lesssim t^{-\frac{\sigma+N}{2\alpha}} \|(-\Delta)^{-\frac{N}2} \widetilde{P_j} P_j f\|_{L^\infty}\\
        \lesssim& t^{-\frac{\sigma+N}{2\alpha}} 2^{-jN} \|(-\Delta)^{-\frac{N}2} \widetilde{P_0}\|_{L^1} \|P_j f\|_{L^\infty}
        \lesssim t^{-\frac{\sigma+N}{2\alpha}} 2^{-jN} \|P_j f\|_{L^\infty}     
    \end{aligned}
    $$
    for $N\geq 0$.
    Hence, we know 
    $$S(t)f\in L^\infty(\mathbb R^n),\quad \forall t>0$$ 
    from the estimate above and the summablity of $\sum_{j\geq 0} 2^{-jN} <\infty$.
    
    Now we choose $j_0$ such that $ 4^{-(j_0+1)\alpha} < t \leq 4^{-j_0\alpha}$, 
    and we choose $N > -\sigma$,
    then for $j_0 > 0$,
    $$t^{-\frac{\sigma}{2\alpha}}\|S(t)f\|_{L^\infty}
    \leq C\sum_{j\leq j_0}2^{(j_0-j)\sigma}\epsilon_j
    + C_N \sum_{j > j_0}2^{(j_0+1-j)(\sigma+N)}\epsilon_j
    \leq \widetilde{C_N} \|\epsilon_j\|_{\ell^\infty(\mathbb Z)}.$$
    Thus 
    $$\sup_{t>0}t^{-\frac{\sigma}{2\alpha}}\|S(t)f\|_{L^\infty} \lesssim \|f\|_{\dot B^{\sigma,\infty}_\infty}.$$

    \textbf{The second inequality}:
    For each $j$, take some $t$ such that $4^{-(j+1)\alpha} < t \leq 4^{-j\alpha}$, 
    $$2^{j\sigma}\|P_j f\|_{L^\infty} = 2^{j\sigma}\|t^{\frac{\sigma}{2\alpha}}S(-t)P_j t^{-\frac{\sigma}{2\alpha}}S(t)f\|_{L^\infty}
    \leq 2^{-\sigma}C\sup_{t>0}\|t^{-\frac{\sigma}{2\alpha}}S(t)f\|_{L^\infty} ,$$
    where the $L^1$ norm for the kernel of $S(-t)P_j$, $\forall t\in (4^{-(j+1)\alpha}, 4^{-j\alpha}]$
    are bounded by $C$ uniformly.
\end{proof}

And note that $\sup_{t>0} t^{1-\frac1{2\alpha}} \|S(t)u\|_{L^\infty} < \infty$  
provided that 
$\|u\|_{\mathcal{E}}<\infty$
by \eqref{equivalence_norm_E},
this implicityly implies $\widehat u(\xi)$ in frequency domain has some decay properties near $0$,
that is, $\lim_{t\to \infty} \|S(t)u\|_{L^\infty} \to 0$.
For $\theta\in C^\infty_c$ such that $\theta(\xi)$ behave like $|\xi|^{4n}$ near $0$,
then we choose some constant $\delta <\frac12 - \frac1{4\alpha}$ and 
compute as follows that 
$$\begin{aligned}
    &\|\theta(tD) u\|_{L^\infty}
\lesssim \|\int \theta(t\xi)e^{t|\xi|^{2\alpha}} e^{ix\cdot\xi}\,d\xi\|_{L^1_x} \|S(t)u\|_{L^\infty}\\
\lesssim& \|\int \theta(t\xi)e^{t|\xi|^{2\alpha}} (1+|x|^2)^{\frac{n}2+\delta} e^{ix\cdot\xi}\,d\xi\|_{L^\infty_x} t^{\frac1{2\alpha}-1}\\
\lesssim& t^{\frac1{2\alpha}-1-n} \left\|\int \theta(\xi)e^{t^{1-2\alpha}|\xi|^{2\alpha}} (1+|tx|^2)^{\frac{n}2+\delta} e^{ix\cdot\xi}\,d\xi\right\|_{L^\infty_x}\\
\lesssim& t^{\frac1{2\alpha}-1+2\delta}\left\||x|^{n+2\delta}\int \theta(\xi)e^{t^{1-2\alpha}|\xi|^{2\alpha}} e^{ix\cdot\xi}\,d\xi\right\|_{L^\infty_x}\\
\lesssim& t^{\frac1{2\alpha}-1+2\delta}\left\|\int \left|(1-\Delta_\xi)^{\frac{n}2+\delta}\left(\theta(\xi)e^{t^{1-2\alpha}|\xi|^{2\alpha}}\right)\,d\xi\right|\right\|_{L^\infty_x}
\lesssim t^{\frac1{2\alpha} -1 + 2\delta} \to 0
\end{aligned}
$$
where the last inequality we use the fact $t > 1$ and $1-2\alpha < 0$.
Then we see that $u\in\mathcal{S}'_h$ from the definition in \cite[Definition 1.26, Page 22]{ref3}.

Therefore, the previously defined space $\mathcal{E}$
is indeed the Besov space $\dot B^{1-2\alpha,\infty}_\infty$,
which completes the proof of Theorem~\ref{equivalence_of_norms}.

\begin{remark}
    A well known result by Bourgain-Pavlovic \cite{ref6} had shown 
    the illposedness of Navier-Stokes when initial data is in $\dot B^{-1, \infty}_\infty$.
    Indeed, this does not contradict to our result here 
    since \eqref{reverse_estimate} is false for $\alpha = 1$.
    Hence we could not expect the wellposedness to be true for $\dot B^{-1, \infty}_\infty$.
\end{remark}

In the following sections,
we denote both $\dot B^{1/\alpha-2,\infty}_\infty$, $\mathcal{E}$ by $\mathcal{B}$
thanks to Lemma~\ref{equivalence_of_norms}.

\section{Preliminaries for some kernel estimates}
As defined in previous section,
the kernel of $S(t)$ is 
$\Phi(t,x) = (t^{\frac1{2\alpha}})^{-n} \Phi(t^{-\frac1{2\alpha}} x)$.
The operator $V$ is the parametrix for the inhomogeneous fractional heat equation with $0$ Cauchy data,
that is, $u = Vf$ if and only if
$$u_t + (-\Delta)^\alpha u = f, u(0) = 0.$$
Then $$(Vf)(t) = \int_0^t S(t-s)f(s)\, ds.$$
Hence the solutions to the fractional heat equation 
$$u_t + (-\Delta)^\alpha u = f,\quad u(x, 0) = u_0(x)$$
are given by
$$u(x,t) = (S(t)u_0)(x) + (Vf)(x,t).$$
Then we intend to discuss the mild solution of the fixed point form generalized Navier-Stokes:
$$u(x,t) = S(t)u_0(x) - (V\nabla \Pi N(u))(x,t).$$

Now we consider the symbol $m$ corresponding to the projection operator $\Pi$
to the divergence free vector fields, 
which is defined by its matrix valued Fourier multiplier $m(D_x)$,
where $$\Pi u(x) = m(D_x)u = \mathcal{F}^{-1}(m(\eta)\widehat{u}(\eta))$$
and its symbol $m$ is given by
$$(m(\eta))_{ij} = \delta_{ij} - \frac{\eta_i\eta_j}{|\eta|^2}.$$
Then the symbol $m$ satisfies the Mihlin-Hormander condition
\begin{equation}\label{Mihlin_Hormander}
    \sup_{\eta\neq 0} |\eta|^{|\alpha|}|\partial_\eta^{\alpha} m(\eta)|\leq C_\alpha
\end{equation}
for all multiindices $\alpha$; 
hence, $m(D_x)$ is a singular integral operator.
\begin{lemma}\label{estimate_Pi_Phi}
    Let $\Phi(x,t)$ as defined before in previous section,
    then combined with the Mihlin-Hormander condition \eqref{Mihlin_Hormander},
    we have the bound that
    $$|(\Pi \Phi)(x)|\leq c(1+|x|)^{-n}.$$
\end{lemma}
\begin{proof}
    By definition, $m(D_x)\Phi(x) = K*\Phi$, where $K = (m)^\lor$.
    We use the Littlewood-Paley decomposition to manipulate:
    $$m(D_x)\Phi = \sum_{k\in\mathbb Z} m(D_x)P_k \Phi = \sum_{k\in\mathbb Z} K_k*\Phi,$$
    where $K_k = (m(\eta)P_k(\eta))^\lor$
    (an abuse of notation here, in $K_k$, $P_k(\eta)$ denotes the Fourier transform of the kernel of $P_k$),
    and here $P_k(\eta)$ is a cut-off (bump) function cutting off to $|\eta|\sim 2^k$.
    Since $P_k f = f * \phi_{2^{-k}}$, where $\phi_{2^{-k}}(x) = 2^{kn}\phi(2^kx)$,
    thus $P_k(\xi) = \widehat{\phi_{2^{-k}}}(\xi) =\widehat{\phi}(2^{-k}\xi) = P_0(2^{-k}\xi)$.
    
    Now we let $\widetilde{m_k}(\eta') = m(2^k \eta')P_k(2^k\eta')$, then $|\eta'|\sim 1$
    and we can observe that $\widetilde{m_k}(\eta') = m(\eta')P_0(\eta') = \widetilde{m_0}(\eta')$, 
    thus $\widetilde{m_k}(\eta')$ also satisfies \eqref{Mihlin_Hormander} with the same constant $C_n$.
    Note that 
    $$|m(D_x)\Phi(x)|
    \leq \sum_{k}\left|\left(m(\xi)P_k(\xi)e^{-|\xi|^{2\alpha}}\right)^\lor(x)\right|.$$
    Let $\Phi_k(x) = \left(m(\xi)P_k(\xi)e^{-|\xi|^{2\alpha}}\right)^\lor(x)$, then
    \begin{equation}\label{change_of_variable_dyadic}
        \begin{aligned}
            \Phi_k(x)
    = \frac1{(2\pi)^n}\int e^{ix\cdot\xi} m(\xi)P_k(\xi)e^{-|\xi|^{2\alpha}}\, d\xi
    =& \frac{2^{kn}}{(2\pi)^n}\int e^{i (2^kx)\cdot\eta} \widetilde{m_k}(\eta) e^{-4^{k\alpha}|\eta|^{2\alpha}}\, d\eta\\
    =& \frac{2^{kn}}{(2\pi)^n}\int e^{i (2^kx)\cdot\eta} \widetilde{m_0}(\eta) e^{-4^{k\alpha}|\eta|^{2\alpha}}\, d\eta
        .\end{aligned}
    \end{equation}
    Moreover, with $N$ (assuming $N$ is even) times integration by parts and applying the Mihlin-Hormander condition \eqref{Mihlin_Hormander},
    we have 
    $$
    \begin{aligned}
    &\left|\int e^{i x\cdot\eta} \widetilde{m_0}(\eta) e^{-c|\eta|^{2\alpha}}\, d\eta\right|
    = (1+|x|^2)^{-N/2}\left|\int (1-\Delta_\eta)^{N/2}\left(e^{i x\cdot\eta}\right) \widetilde{m_0}(\eta) e^{-c|\eta|^{2\alpha}}\, d\eta\right| \\
    =& (1+|x|^2)^{-N/2} \int \left|(1-\Delta_\eta)^{N/2}\left(\widetilde{m_0}(\eta) e^{-c|\eta|^{2\alpha}}\right)\right|\, d\eta\\
    \leq& C_N (1+|x|)^{-N} \sum_{|\gamma| + |\beta|\leq N} \int |\partial^\beta \widetilde{m_0}(\eta)| |\partial^\gamma e^{-c|\eta|^{2\alpha}}|\, d\eta\\
    \lesssim&_\alpha C_N (1+c)^N (1+|x|)^{-N} \sum_{|\gamma| + |\beta|\leq N} \int_{|\eta|\sim 1}|\eta|^{-|\beta|}(1+|\eta|)^{|\gamma|(2\alpha-1)} e^{-c|\eta|^{2\alpha}}\, d\eta\\
    \lesssim&_\alpha C_N (1+c)^N (1+|x|)^{-N} e^{-c}.\\
    \end{aligned}
    $$
    Then use the estimate above, we could give a bound for $\Phi_k(x)$ in \eqref{change_of_variable_dyadic}:
    $$|\Phi_k(x)| \leq C_N 2^{kn} (1+|2^kx|)^{-N} \frac{(1+4^{k\alpha})^N}{e^{4^{k\alpha}}}.$$
    Then for $\forall k$, there exists some constant such that 
    $$|\Phi_k(x)| \lesssim C_N 2^{kn} (1+|2^kx|)^{-N}$$
    and $\exists K_N$ sufficient large, $\forall k > K_N$,
    $$|\Phi_k(x)| \leq C_N 2^{kn} (1+|2^kx|)^{-N} 4^{-kn} \leq C_N 2^{-kn},$$
    since $\frac{(1+x)^N x^{\frac{n}{\alpha}}}{e^x}\to 0$ as $x\to +\infty$.

    Now we take a fixed $N > n$ and let $k_0\in\mathbb Z^-$ such that $2^{-k_0-1}\leq |x|\leq 2^{-k_0}$, then  
    $$
    \begin{aligned}
        |m(D_x)\Phi(x)| \leq C_N \sum_{k\in\mathbb Z} 2^{kn} (1+|2^kx|)^{-N}
    \lesssim&_N  \sum_{k\leq k_0} 2^{kn} +  |x|^{-N} \sum_{k> k_0} 2^{k(n-N)}\\
    \lesssim&_N  2^{k_0n} + 2^{k_0(n-N)} |x|^{-N}
    \lesssim_N |x|^{-n}
    \end{aligned}
    $$
    when $|x|>1$.
    And for $|x| \leq 1$,
    we have
    $$|m(D_x)\Phi(x)| \leq C_N \left(\sum_{k > K_N} 2^{-kn} + \sum_{k \leq K_N} 2^{kn}\right)
    \lesssim_{N,n} 1
    .$$
    Combining the cases for $|x|>1$ and for $|x|\leq 1$, we have 
    $$|m(D_x)\Phi(x)| \lesssim (1+|x|)^{-n}.$$
\end{proof}

\begin{lemma}\label{estimate_Pi_Phi_scale}
    From Lemma~\ref{estimate_Pi_Phi}, 
    scaling shows that the kernel function $k_t(x)=\Pi\Phi_{t^{\frac1{2\alpha}}}$ of $\Pi S(t)$ satisfies
    $$|k_t(x)|\leq C \left(t^\frac1{2\alpha}+|x|\right)^{-n}.$$
\end{lemma}
\begin{lemma}\label{estimate_Pi_nabla_Phi}
    Similarly, we have bounds for the kernel of $\Pi\nabla S(t)$ that 
    $$|m(D_x)\nabla\Phi_{t^{\frac1{2\alpha}}}(x)|\leq C \left(t^\frac1{2\alpha}+|x|\right)^{-n-1}.$$
\end{lemma}

\begin{lemma}\label{integrability_fractional_temperature}
    The kernel $k(t,x)$ of $\left(t(-\Delta)^{\alpha}\right)^{\frac{N}{2\alpha}} S(t)$ 
    is integrable and its $L^1$ norm is uniformly bounded.
\end{lemma}
\begin{proof}
    By scaling property, it suffices to show $\|k(1,x)\|_{L^1}<\infty$.
    Let 
    $$
    \begin{aligned}
        K_j(x) &= (-\Delta)^{N/2} P_j\Phi(x)
        = \frac1{(2\pi)^n}\int e^{ix\cdot\xi} |\xi|^N P_j(\xi)e^{-|\xi|^{2\alpha}}\, d\xi\\
        &= \frac{2^{j(n+N)}}{(2\pi)^n}\int e^{i (2^j x)\cdot\xi} |\xi|^N P_0(\xi) e^{-4^{j\alpha}|\xi|^{2\alpha}}\, d\xi     
    .\end{aligned}
    $$
    Since
    $$
    \begin{aligned}
    &\left|\int e^{i x\cdot\eta} |\eta|^N P_0(\eta) e^{-c|\eta|^{2\alpha}}\, d\eta\right|
    \leq C_M (1+|x|)^{-M} \int \left|(1-\Delta_\eta)^{M/2}\left(|\eta|^N P_0(\eta) e^{-c|\eta|^{2\alpha}}\right)\right|\, d\eta\\
    \lesssim& C_M (1+|x|)^{-M} \sum_{|\delta|\leq M} \int_{|\eta|\sim 1} |\partial^\delta e^{-c|\eta|^{2\alpha}}|\, d\eta
    \lesssim_\alpha C_M (1+c)^M (1+|x|)^{-M} e^{-c}
    ,\end{aligned}
    $$
    we have  
    $$|K_j(x)|\leq C_M 2^{j(n+N)}(1+|2^j x|)^{-M} (1+4^{j\alpha})^M e^{-4^{j\alpha}}.$$
    Then for $|x|\sim 2^{-j_0}$, $M>n+N$, we have 
    $$|k(1,x)| = |(-\Delta)^{N/2} \Phi(x)|
    \leq \sum_{j\in\mathbb Z} |K_j(x)|
    \leq C_M (\sum_{j \leq j_0} 2^{j(n+N)}  + \sum_{j>j_0} 2^{j(n+N-M)}|x|^{-M}) \lesssim |x|^{-n-N}
    .$$
    And there exists $J$ such that $\forall j> J$, such that   
    $|K_j(x)|\leq C_M 2^{-j}(1+|2^j x|)^{-M}$
    Then 
    $$|k(1,x)| 
    \leq \sum_{j\in\mathbb Z} |K_j(x)|
    \leq C_M \sum_{j\leq J}2^{j(n+N)} + \sum_{j>J} 2^{-j} \lesssim 1.$$
    Hence $|k(1,x)|\lesssim (1+|x|)^{-n-N}$.
    In particular, $k(1,x)$ is integrable.
\end{proof}

Now we devote a full section to the proof of Theorem~\ref{main_theorem}.

\section{Proof of Theorem~\ref{main_theorem} and Corollary~\ref{corollary_main_theorem}}
Let 
$$\Psi(u) = S(t)u_0(x) - (V\nabla \Pi N(u))(x,t),$$
then we solve fixed point form of generalized Navier-Stokes:
\begin{equation}
    \Psi u = u.
\end{equation}
For small initial data we want to solve this in $GX$ using a fixed point argument. 
Since $N$ is quadratic, the small Lipschitz constant follows for small initial data 
if the nonlinearity has the correct mapping properties. 
Hence the result is a consequence of the following two lemmas:
\begin{lemma}\label{lemma1}
    $N$ maps $GX$ into $GY$.
\end{lemma}
\begin{lemma}\label{lemma2}
    $V\nabla\Pi$ maps $GY$ into $GX$.
\end{lemma}
Since $V\nabla\Pi$ is linear, and from Lemma~\ref{lemma2},
$\|V\nabla\Pi f\|_{GX} \leq C \|f\|_{GY}$.
We assume $\|u_0\|_{\mathcal{B}}\leq \epsilon$.
Then for $u, v\in B_{GX}(0,2\varepsilon C)$,
$$\|\Psi(u)\|_{GX} \leq \|S(t)u_0(x)\|_{GX} + \|V\nabla \Pi N(u)\|_{GX}
\leq \|u_0\|_{\mathcal{B}}+  C\|u\|^2_{GX}
\leq \epsilon + C\|u\|^2_{\mathcal{B}}
.$$
And
$$\|\Psi(u)-\Psi(v)\|_{GX}\leq C\|N(u)-N(v)\|_{GY} 
\leq C\|u-v\|_{GX}(\|u\|_{GX} + \|v\|_{GX})
\leq 2\varepsilon C \|u-v\|_{GX}
.$$
Thus for $\epsilon$ sufficiently small, namely $\epsilon_n\leq \frac{1}{4C^2 + 1}$,
$\Psi$ maps to $B_{GX}(0,2\varepsilon C)$ to itself and is a contraction mapping.
Therefore, we apply the fixed point theorem and this completes the proof of Theorem~\ref{main_theorem}.

Now, we turn to the proof of these two lemmas. 
The proof of Lemma~\ref{lemma1} is obvious and straightforward.
Then for Lemma~\ref{lemma2}, we prove as follows.
\begin{proof}[Proof of Lemma~\ref{lemma2}]
    It suffices to prove the pointwise estimate
    \begin{equation}\label{pointwise_estimate}
        |V\nabla\Pi f(x,t)|\leq ct^{-(1-\frac1{2\alpha})}\|f\|_{GY}.
    \end{equation}

    \textbf{Step 1.Scaling.} 
    We claim that this estimate is scale invariant and translation invariant,
    which is the motivation of the definition of function spaces.
    Now we check this property:
    For \eqref{pointwise_estimate}, 
    suppose we have proved the result for $t = 1, x = 0$,
    then apply the result to $g(x,t) = f(\lambda^{\frac1{2\alpha}}x, \lambda t)$, 
    by a similar argument to \eqref{scaling_L_2_average}, we have
    $\|g\|_{GY} = \lambda^{-(2-1/\alpha)} \|f\|_{GY}$
    and 
    $$
    \begin{aligned}
        &V\nabla\Pi g(0,1)\\
    =& \int_0^1\int_{\mathbb R^n} \Phi(s,y) \nabla\Pi g(-y, 1-s)\, dy\, ds\\
    =& \lambda^{\frac1{2\alpha}} \int_0^1 \int_{\mathbb R^n} \Phi(s,y) (\nabla\Pi f)(\lambda^{\frac1{2\alpha}}(-y), \lambda(1-s))\, dy\, ds\\
    =& \lambda^{\frac1{2\alpha}} \lambda^{\frac{n}{2\alpha}} \int_0^\lambda \int_{\mathbb R^n} \Phi(s,y) (\nabla\Pi f)((-y), \lambda-s))\, d(\lambda^{-\frac1{2\alpha}}y)\, d(\frac{s}{\lambda})\\
    =& \lambda^{\frac1{2\alpha}-1} \int_0^\lambda \int_{\mathbb R^n} \Phi(s,y) (\nabla\Pi f)((-y), \lambda-s))\, dy\, ds
    .\end{aligned}
    $$
    Thus take $\lambda = t$, then the estimate for $(0,t)$ follows from the above scale invariance property.
    And the translation invariance is obvious.
    
    Thus it suffices to prove 
    \begin{equation}\label{step_1_formula_1}
        |V\nabla\Pi f(0,1)|\leq c\|f\|_{GY}.
    \end{equation}

    \textbf{Step 2: Localization.}
    Let $\chi$ be the characteristic function of $B(0,2)\times[0,1]$.
    Then $f = \chi f + (1-\chi) f$.
    Clearly both components are still in $GY$.
    Since for $g\in GY$,
    $$\Pi_x g(x-y) = \mathcal{F}_x^{-1} \left(\left(\delta_{ij}-\frac{\xi_i\xi_j}{|\xi|^2}\right)\widehat g(\xi)e^{-iy\cdot\xi}\right),$$
    then we calculate by Parseval indentity,
    \begin{equation}\label{Pi_nabla_change_position}
        \begin{aligned}
            &\int_{\mathbb R^n} h(y)\Pi_x g(x-y)\, dy
         = \int_{\mathbb R^n} \widehat h(\xi) \left(\delta_{ij}-\frac{(-\xi_i)(-\xi_j)}{|-\xi|^2}\right)\widehat g(-\xi)e^{ix\cdot(-\xi)}\, d\xi\\
         =& \int_{\mathbb R^n} \left(\delta_{ij}-\frac{\xi_i\xi_j}{|\xi|^2}\widehat h(\xi)\right) \widehat g(-\xi)e^{-ix\cdot\xi}\, d\xi
         = \int_{\mathbb R^n} \left(\Pi_y h(y)\right) g(x-y)\, dy    
         .\end{aligned}
    \end{equation}
    Then we know from integration by parts and applying the equation above to $\nabla\Phi$ and $g$,
    $$
    V\nabla\Pi g(x,t) 
    = \int_{\mathbb R^n} \int_0^t \Phi(s,y) \nabla\Pi g(x-y, t-s)\, ds\, dy
    = \int_{\mathbb R^n} \int_0^t \Pi\nabla \Phi(s,y) g(x-y, t-s)\, ds\, dy
    .$$
    Thus the kernel $K$ of $V\nabla\Pi$ is $K = \Pi\nabla\Phi_{s^{\frac1{2\alpha}}}(y)$,
    then from Lemma~\ref{estimate_Pi_nabla_Phi}, we know that 
    \begin{equation}\label{step_1_kernel_bound}
        |K(x,t)|\leq C(t^{\frac1{2\alpha}}+|x|)^{-n-1}.
    \end{equation}
    Then we claim that 
    \begin{equation} \label{step_1_cutoff}
        \|V\nabla\Pi (1-\chi)f\|_{L^\infty(Q(0,1))} 
        \leq C \sup_{x\in\mathbb R^n} \int_{Q(x,1)} |f|\, dy\, dt .
    \end{equation}
    Take $(x_0, t_0)\in Q(0,1)$, then 
    $$
    \begin{aligned}
        &|V\nabla\Pi (1-\chi)f(x_0, t_0)|\\
    \leq& C\int_{\mathbb R^n\setminus B(0,2)}\int_0^1 \left((t_0-s)^{\frac1{2\alpha}}+|x_0-y|\right)^{-(n+1)} |f(y,s)|\, ds\, dy\\
    \leq& C\sum_{k\in(\frac1{\sqrt{n}}\mathbb Z)^n, |k|>2} \int_{B(k,1)}\int_0^1 \frac1{|x_0-y|^{n+1}} |f(y,s)|\, ds\, dy\\
    \leq& C \sup_{x\in\mathbb R^n} \int_{Q(x,1)} |f|\, dy\, dt  \sum_{k\in(\frac1{\sqrt{n}}\mathbb Z)^n, |k|> 2} \frac1{(|k|-2)^{n+1}}\\
    =& C \sup_{x\in\mathbb R^n} \int_{Q(x,1)} |f|\, dy\, dt \left(\widetilde C + \sum_{k\in(\frac1{\sqrt{n}}\mathbb Z)^n, |k|> 4} \frac1{(|k|-2)^{n+1}}\right)\\
    \leq& C \sup_{x\in\mathbb R^n} \int_{Q(x,1)} |f|\, dy\, dt \left( \widetilde C + \int_{|z|>3} \frac1{(|z|-2)^{n+1}}\, dz \right)
    = C' \sup_{x\in\mathbb R^n} \int_{Q(x,1)} |f|\, dy\, dt.\\
    \end{aligned}
    $$
    Thus we have verified the claim \eqref{step_1_cutoff}.
    Actually, the claim \eqref{step_1_cutoff} is much stronger than 
    \eqref{step_1_formula_1} under the assumption that
    $f$ supported outside $B(0,2)\times [0,1]$ thanks to the following observation:
    \begin{equation}\label{GY_L2_estimate}
        \begin{aligned}
            R^{-(n+2-2\alpha)}\int_{Q(x,R)} |f(y,t)|\, dt\, dy
            \lesssim \frac1{R^{2-2\alpha}}\int_0^{R^{2\alpha}} t^{-\left(2-\frac1{\alpha}\right)}\, dt  \sup_{t>0} t^{2-\frac1{\alpha}}\|f(\cdot,t)\|_{L^\infty}&\\
            = \sup_{t>0} t^{\left(2-\frac1{\alpha}\right)}\|f(\cdot,t)\|_{L^\infty} = \|f\|_{GY}.&     \\
        \end{aligned}
    \end{equation}
    Hence, it suffices to look now at $\chi f$;
    namely, without any restriction in generality, 
    we can and do assume in the sequel that $f$ is supported in $B(0, 2)\times [0, 1]$.
    
    \textbf{Step 3: The pointwise estimate.}
    We intend to show the pointwise estimate \eqref{step_1_formula_1} 
    when $f$ is supported in $B(0, 2)\times[0, 1]$.
    Actually, it follows easily from the kernel bound \eqref{step_1_kernel_bound}.
    Since 
    $$\begin{aligned}
        |V\nabla\Pi f(0,1)|
        \leq& \int_{\mathbb R^n}\int_0^{\frac12} \left((1-t)^{\frac1{2\alpha}}+|y|\right)^{-(n+1)} |f(y, t)| \, dt\, dy\\
        +& \int_{\mathbb R^n}\int_{\frac12}^1 \left((1-t)^{\frac1{2\alpha}}+|y|\right)^{-(n+1)} |f(y, t)| \, dt\, dy       
    .\end{aligned}
    $$
    For the part of $f$ in $B(0, 2)\times [0, \frac12]$ we can use the $L^1$ bound on $f$
    combined with the boundedness of the kernel away from $0$, that is,
    $$\left|\int_{B(0,2)}\int_0^{\frac12} \left((1-t)^{\frac1{2\alpha}}+|y|\right)^{-(n+1)} f(y, t) \, dt\, dy\right|
    \lesssim \int_{B(0,2)}\int_0^{\frac12} |f(y,t)|\, dt\, dy 
    \lesssim \|f\|_{GY},$$
    where we use the observation \eqref{GY_L2_estimate} in Step 2 here.
    For the part of $f$ in $B(0, 2)\times[\frac12 , 1]$ we can use the $L^\infty$ bound on $f$ combined with the
    integrability of the kernel at $0$, that is,
    $$
    \begin{aligned}
        &\left|\int_{\mathbb R^n}\int_{\frac12}^1 \left((1-t)^{\frac1{2\alpha}}+|y|\right)^{-(n+1)} f(y, t)\, dt\, dy  \right|\\
    \leq& \|f\|_{L^\infty(\mathbb R^n\times [\frac12, 1])} \int_{B(0,2)}\int_{\frac12}^1 \left((1-t)^{\frac1{2\alpha}}+|y|\right)^{-(n+1)} \, dt\, dy\\
    \lesssim& \|f\|_{GY} \int_{B(0,2)}\int_{\frac12}^1 \left((1-t)^{\frac1{2\alpha}}+|y|\right)^{-(n+1)}\, dt\, dy
    .\end{aligned}
    $$
    Then it suffices to check the integrability of $\left((1-t)^{\frac1{2\alpha}}+|y|\right)^{-(n+1)}$ near $0$:
    $$\begin{aligned}
        \int_0^1\left((1-t)^{\frac1{2\alpha}}+|y|\right)^{-(n+1)}\, dt
        =& 2\alpha\int_0^1 \frac{s^{2\alpha-1}}{\left(s+|y|\right)^{n+1}}\, ds\\
        = 2\alpha |y|^{-(n+1-2\alpha)} \int_0^{|y|} \frac{s^{2\alpha-1}}{(s+1)^{n+1}}\, ds
        \leq& 2\alpha B(2\alpha, n+1-2\alpha) |y|^{-(n+1-2\alpha)} 
    ,\end{aligned}
    $$
    which is integrable near $y = 0 \in\mathbb R^n$ when $\alpha\in(\frac12, 1)$.

    Here we complete the proof of \eqref{step_1_formula_1}.
    And this completes the proof of Theorem~\ref{main_theorem}.
\end{proof}
To conclude, we prove the Corollary~\ref{corollary_main_theorem} at the end.
\begin{proof}[Proof of Corollary~\ref{corollary_main_theorem}]
    Suppose $u(x,t)$ is the unique solution constructed above for small initial data $u_0(x)\in\mathcal{B}$.
    
    For $t \leq t_0$, we have 
    $$\sup_{0 < t\leq t_0} t^{1-\frac1{2\alpha}}\|S(t)u(\cdot,t_0)\|_{L^\infty}
    \lesssim t_0^{1-\frac1{2\alpha}} \|u(\cdot,t_0)\|_{L^\infty} \leq \|u\|_{GX},$$
    where we use the integrability of the kernel of $S(t)$.
    For $t > t_0$, since 
    $$u(x,t) = S(t-t_0)u(x, t_0) - \int_{t_0}^t S(t-s) (\nabla\Pi N(u))(x,s)\, ds,$$
    which is equivalent to 
    \begin{equation}\label{integral_form_NS_t0}
        \begin{aligned}
        S(t)u(x, t_0) 
        &= S(t_0)u(x,t) + S(t_0)\int_0^t S(t-s) (\nabla\Pi Nu)(x,s)\, ds - S(t)\int_0^{t_0} S(t_0-s) (\nabla\Pi Nu)(x,s)\, ds \\
        &= S(t_0)u(x,t) + S(t_0) (V\nabla\Pi N u)(x, t) - S(t) (V\nabla\Pi N u)(x, t_0)
        .\end{aligned}
    \end{equation}
    From Theorem~\ref{main_theorem} and Lemma~\ref{lemma2}, 
    we know $u(x,t)\in GX,\quad (V\nabla\Pi N u)(x, t)\in GX$, respectively.
    Then applying \eqref{integral_form_NS_t0} and the fact that the kernel $\Phi(t_0,x)$ of $S(t_0)$ is integrable, we have 
    \begin{equation}
       \label{estimate_for_large_time}
       \sup_{t > t_0} t^{1-\frac1{2\alpha}}\|S(t)u(\cdot,t_0)\|_{L^\infty}
       \lesssim \|u\|_{GX} + \|V\nabla\Pi N u\|_{GX} + \sup_{t>t_0}t^{1-\frac1{2\alpha}}\|S(t) (V\nabla\Pi N u)(\cdot, t_0)\|_{L^\infty}    
    .\end{equation}
    Then we compute 
    $$\begin{aligned}
        &t^{1-\frac1{2\alpha}}\|S(t) (V\nabla\Pi N u)(\cdot, t_0)\|_{L^\infty}  \\
        \leq& t^{1-\frac1{2\alpha}}\|\int_0^{t_0} (\Pi\nabla S(t+t_0-s)) (u(s)\otimes u(s))\, ds\|_{L^\infty}\\
        \leq& t^{1-\frac1{2\alpha}}\left\|\int_0^{t_0}\int_{\mathbb R^n} \frac{|u(y,s)|^2}{\left((t+t_0-s)^{\frac1{2\alpha}}+|x-y|\right)^{n+1}}\, dy\, ds\right\|_{L^\infty_x}    \\   
        \leq& t^{1-\frac1{2\alpha}}\left\|\int_0^{t_0}\int_{\mathbb R^n} \frac1{s^{2-\frac1{\alpha}}\left((t+t_0-s)^{\frac1{2\alpha}}+|x-y|\right)^{n+1}}\, dy\, ds\right\|_{L^\infty_x} \|u\|_{GX}^2   \\   
        \leq& t^{1-\frac1{2\alpha}} \left(\int_0^{t_0}\int_{\mathbb R^n} \frac1{s^{2-\frac1{\alpha}}\left(t^{\frac1{2\alpha}}+|y|\right)^{n+1}}\, dy\, ds \right) \|u\|_{GX}^2   \\   
        \lesssim&_{\alpha} \frac{t^{1-\frac1{2\alpha}}}{t_0^{1-\frac1{\alpha}}} \left(\int_{\mathbb R^n} \frac1{\left(t^{\frac1{2\alpha}}+|y|\right)^{n+1}}\, dy \right) \|u\|_{GX}^2   \\   
        \lesssim&_{\alpha} \left(\frac{t_0}{t}\right)^{\frac{1-\alpha}{\alpha}} \left(\int_{\mathbb R^n} \frac1{\left(1+|y|\right)^{n+1}}\, dy \right) \|u\|_{GX}^2    
        \lesssim_{\alpha} \|u\|_{GX}^2, \\
    \end{aligned}
    $$
    where the first inequality is the similar fact to \eqref{Pi_nabla_change_position},
    the second is by Lemma~\ref{estimate_Pi_nabla_Phi}, the fifth and last inequality follow from the choice of $\alpha\in (\frac12, 1)$ and $t>t_0$.
    Then combined this result with \eqref{estimate_for_large_time}, we have 
    $$ 
    \sup_{t > t_0} t^{1-\frac1{2\alpha}}\|S(t)u(\cdot,t_0)\|_{L^\infty}
    \lesssim \|u\|_{GX} + \|V\nabla\Pi N u\|_{GX} + \|u\|_{GX}^2 
    \lesssim \|u\|_{GX} + \|u\|_{GX}^2    
    .$$
    Combining the two cases above, and we know 
    $$\sup_{t > 0} t^{1-\frac1{2\alpha}}\|S(t)u(\cdot,t_0)\|_{L^\infty} 
    \lesssim \|u\|_{GX} + \|u\|_{GX}^2 
    ,$$
    which completes the proof.
\end{proof}

\section{Further results about regularities}
Motivated by \cite{ref4}, we study the regularity about the solution 
constructed by Theorem~\ref{main_theorem}.
To state the results, we introduce some new notations.
\begin{definition}
    For any nonnegative integer $k$, we introduce the space $GX^k$ which equipped with the norm 
    $$\|u\|_{GX^k} = \sup_{\alpha_1 +\cdots+\alpha_n = k} \sup_{t>0} t^{(1-\frac1{2\alpha})+\frac{k}{2\alpha}} \|\partial_{x_1}^{\alpha_1}\cdots\partial_{x_n}^{\alpha_n}u(\cdot,t)\|_{L^\infty}.$$
    And note that $GX^0$ is the same as the definition for $GX$ above.
    For simplicity, we denote 
    $$\nabla^k u = \partial_{x_1}^{\alpha_1}\cdots\partial_{x_n}^{\alpha_n}u,\quad \alpha_1 +\cdots+\alpha_n = k .$$
    Then 
    $$\|u\|_{GX^k} = \sup_{t>0} t^{(1-\frac1{2\alpha})+\frac{k}{2\alpha}} \|\nabla^k u\|_{L^\infty}.$$
\end{definition}

\begin{theorem}\label{decay_theorem}
    For small initial data $\|u_0\|_{\mathcal{B}} < \epsilon$,
    the solution constructed in Theorem~\ref{main_theorem} satisfies 
    $$t^{(1-\frac1{2\alpha})+\frac{k}{2\alpha}} \nabla^k u \in GX^0$$ for any $k\geq 0$.
    And this is equivalent to $u\in GX^k$.
\end{theorem}

Similar to Lemma~\ref{estimate_Pi_nabla_Phi}, we have the general version that 
\begin{lemma}\label{estimate_Pi_nabla_k_Phi}
    Similarly, we have bounds for the kernel of $\Pi\nabla^{k+1} S(t)$ that 
    $$|\nabla^{k+1}\Pi\Phi_{t^{\frac1{2\alpha}}}(x)|\leq C \left(t^\frac1{2\alpha}+|x|\right)^{-(n+k+1)}.$$
\end{lemma}

In fact, we need a more precise estimate for the proof of the analyticity result:
\begin{lemma}\label{precise_estimate_Pi_nabla_k_Phi}
    The kernel of $\Pi\nabla^{k+1} S(t)$ satisfies that
    $$|\nabla^{k+1}\Pi\Phi_{t^{\frac1{2\alpha}}}(x)|\leq C^k k^{\frac{k}{2\alpha}}t^{-\frac{k}{2\alpha}} \left(\left(\frac{t}{k}\right)^\frac1{2\alpha}+|x|\right)^{-(n+1)}$$
    for all $t>0, x\in\mathbb R^n$, $\nabla^{k+1}$ denote $\partial^{\beta_1}_{x_1}\cdots \partial^{\beta_n}_{x_n}\nabla$ where $\beta_1+\cdots+\beta_n = k$.
\end{lemma}
Before stating the proof, we need the following fact cited from \cite{ref4}.
\begin{lemma}\label{lemma_cited}
    There exists some constant $C$ such that 
    $$\int_{\mathbb R^n} (a+|x-y|)^{-n-1}(b+|y|)^{-n-1}\, dy \leq Ca^{-1}(a+|x|)^{-n-1}$$
    for all $x\in\mathbb R^n$ and $0<a<b$.
\end{lemma}

\begin{proof}[Proof of Lemma~\ref{precise_estimate_Pi_nabla_k_Phi}]
    By scaling property, it suffices to prove 
    $$|\nabla^{k+1}\Pi\Phi(x)|\leq C^k k^{\frac{k}{2\alpha}} \left(k^{-\frac1{2\alpha}}+|x|\right)^{-(n+1)}.$$
    Since 
    $$\nabla^{k+1}\Pi\Phi(x) = \nabla\Pi\Phi(\frac12, \cdot) * \nabla^{k}\Phi(\frac12, \cdot),$$
    where 
    \[
    |\nabla\Pi\Phi(\frac12, x)|\leq C(1+|x|)^{-n-1}.
    \]
    Then by Lemma~\ref{lemma_cited}, it suffices to prove 
    $$|\nabla^{k}\Phi(\frac12, x)|\leq C^{k-1} k^{\frac{k-1}{2\alpha}} \left(k^{-\frac1{2\alpha}}+|x|\right)^{-(n+1)}.$$
    From Lemma~\ref{integrability_fractional_temperature}, then 
    $|\nabla\Phi(\frac12, x)|\leq C(1+|x|)^{-n-1}$
    and then scaling implies
    \begin{equation}\label{induction_equation}
        |\nabla\Phi(\frac1{2k}, x)|\leq Ck^{\frac{n+1}{2\alpha}}(1+|k^{\frac1{2\alpha}}x|)^{-n-1} = C(k^{-\frac1{2\alpha}}+|x|)^{-n-1}
    .\end{equation}
    For fixed $k > 0$, we write
    $$\nabla^{l+1}\Phi(\frac1{2k}, x) = \nabla \Phi(\frac1{2k},\cdot) * \nabla^l \Phi(\frac1{2k},\cdot),$$
    then we could use induction on $l \geq 0$ by applying Lemma~\ref{lemma_cited} to get the estimate 
    $$
    |\nabla^{l+1}\Phi(\frac1{2k}, x)| \leq C^l k^{\frac{l}{2\alpha}}(k^{-\frac1{2\alpha}}+|x|)^{-n-1}
    ,$$
    where \eqref{induction_equation} is our induction hypothesis.
    Let $l = k$, then we get our desired result.
\end{proof}

And from Lemma~\ref{integrability_fractional_temperature}, we know that 
the kernel of $\nabla S(t)$ is integrable, thus we have the trivial estimate that 
\begin{equation}\label{estimate_nabla_fractional_temperature}
    \|\nabla S(t)u\|_{L^\infty}\leq \frac{C}{t^{\frac1{2\alpha}}}\|u\|_{L^\infty}.
\end{equation}

\smallskip
Now we state two lemmas, which are the constituents of the proof of Theorem~\ref{decay_theorem}.
We first manipulate the linear terms.
\begin{lemma}\label{lemma1_decay}
    For $k\geq 0$, we have 
    $$\|S(t)u_0\|_{GX^k} \leq C_k \|u_0\|_{\mathcal{B}}.$$
\end{lemma}
\begin{proof}
    By the equivalence of norms that we have proved previously in the precise version \eqref{precise_version_of_lemma} of Lemma~\ref{Besov_norm_equivalence},
    we know that 
    $$t^{(1-\frac1{2\alpha})+\frac{k}{2\alpha}} \|S(t)\nabla^k u_0\|_{L^\infty}
    = \|\nabla^k u_0\|_{\dot B^{-k-(2\alpha-1),\infty}_\infty}.$$
    Then this lemma follows from the fact that 
    $\nabla$ is bounded from $\dot B^{-l,\infty}_\infty$ to $\dot B^{-l-1,\infty}_\infty$
    for $l\geq 0$.
\end{proof}
Then for the nonlinear terms, we have 
\begin{lemma}\label{lemma2_decay}
    For $k\geq 1$, we have
    $$\|B(u,v)\|_{GX^k}
    \leq C_{0}(k)\|u\|_{GX^0}\|v\|_{GX^0} + C_1\|u\|_{GX^0}\|v\|_{GX^k} + C_1\|u\|_{GX^k}\|v\|_{GX^0} 
    + C_2\sum_{l=1}^{k-1} \left(\begin{matrix}k \\ l\end{matrix}\right)\|u\|_{GX^l} \|v\|_{GX^{k-l}}$$
    where $$B(u,v)(x,t) = V\Pi\nabla(u\otimes v) = \int_0^t S(t-s)\Pi\nabla(u(\cdot,s)\otimes v(\cdot,s)\, ds.$$
    And for $k = 0$, we have proved that $\|B(u,v)\|_{GX^0} \leq C\|u\|_{GX^0}\|v\|_{GX^0}$.
\end{lemma}
\begin{proof}
    Fix some $m = m(k)$ which will be determined later.
    If $0< s < t(1-\frac1{m})$, we use Lemma~\ref{precise_estimate_Pi_nabla_k_Phi} to obtain 
    \begin{equation}\label{lemma2_decay_estimate_1}
        \begin{aligned}
            &\int_0^{t(1-\frac1{m})} |\nabla^k S(t-s)\Pi\nabla u(s)\otimes v(s)|\, ds\\
        \lesssim& \int_0^{t(1-\frac1{m})} C^k k^{\frac{k}{2\alpha}}\int_{\mathbb R^n} \frac{|u(y,s)| |v(y,s)|}{(t-s)^{\frac{k}{2\alpha}}\left((\frac{t-s}{k})^{\frac1{2\alpha}}+|x-y|\right)^{n+1}}\, dy\, ds\\
        \lesssim& \int_0^{t(1-\frac1{m})} C^k k^{\frac{k}{2\alpha}}\int_{\mathbb R^n} \frac{|u(y,s)| |v(y,s)|}{(t-s)^{\frac{k+n+1}{2\alpha}}\left((\frac1{k})^{\frac1{2\alpha}}+\frac{|x-y|}{(t-s)^{\frac1{2\alpha}}}\right)^{n+1}}\, dy\, ds\\
        \lesssim& C^k k^{\frac{k}{2\alpha}}(\frac{m}{t})^{\frac{k+n+1}{2\alpha}} \int_0^{t(1-\frac1{m})} \int_{\mathbb R^n} \frac{|u(y,s)| |v(y,s)|}{\left(\frac1{k^{\frac1{2\alpha}}}+\frac{|x-y|}{t^{\frac1{2\alpha}}}\right)^{n+1}}\, dy\, ds\\
        \lesssim& C^k k^{\frac{k}{2\alpha}}(\frac{m}{t})^{\frac{k+n+1}{2\alpha}} \int_0^{t(1-\frac1{m})} \sum_{q\in\mathbb Z^d} \frac1{(\frac1{k^{\frac1{2\alpha}}}+|q|)^{n+1}}\int_{y\in x+B(t^{\frac1{2\alpha}}q, t^{\frac1{2\alpha}})} |u(y,s)| |v(y,s)|\, dy\, ds\\
        \lesssim& C^k k^{\frac{k+1}{2\alpha}}m^{\frac{k+n+1}{2\alpha}} \frac1{t^{\frac{k+2\alpha-1}{2\alpha}}} \frac1{t^{\frac{n+2-2\alpha}{2\alpha}}} \sup_{q\in\mathbb Z^n}\int_0^{t(1-\frac1{m})}\int_{y\in x+B(t^{\frac1{2\alpha}}q, t^{\frac1{2\alpha}})} |u(y,s)| |v(y,s)|\, dy\, ds\\
        \lesssim& C^k k^{\frac{k+1}{2\alpha}}m^{\frac{k+n+1}{2\alpha}}\frac1{t^{\frac{k+2\alpha-1}{2\alpha}}} \|u\|_{GX^0} \|v\|_{GX^0}
        .\end{aligned}
    \end{equation}
    
    If $t(1-\frac1{m}) \leq s < t$, we use \eqref{estimate_nabla_fractional_temperature} to obtain 
    \begin{equation}
        \begin{aligned}
            |\nabla^kS(t-s)\Pi\nabla(u\otimes v)|
            \lesssim& \frac1{(t-s)^{\frac1{2\alpha}}} \sum_{l=0}^k \left(\begin{matrix}k \\ l\end{matrix}\right) \|\nabla^l u(\cdot, s)\|_{L^\infty} \|\nabla^{k-l} v(\cdot, s)\|_{L^\infty}  \\
            \lesssim& \frac1{(t-s)^{\frac1{2\alpha}}} \sum_{l=0}^k s^{-\frac{k+4\alpha-2}{2\alpha}}\left(\begin{matrix}k \\ l\end{matrix}\right) \|u\|_{GX^l} \|v\|_{GX^{k-l}}
        .\end{aligned}
    \end{equation}
    Therefore, 
    \begin{equation}\label{lemma2_decay_estimate_2}
        \left|\int_{t(1-\frac1m)}^t \nabla^kS(t-s)\Pi\nabla(u(s)\otimes v(s))\, ds\right|
    \lesssim \int_{t(1-\frac1m)}^t \frac1{(t-s)^{\frac1{2\alpha}}}s^{-\frac{k+4\alpha-2}{2\alpha}}\, ds \sum_{l=0}^k \left(\begin{matrix}k \\ l\end{matrix}\right)\|u\|_{GX^l} \|v\|_{GX^{k-l}}
    .\end{equation}
    Let 
    $$I(k,m,t) = \int_{t(1-\frac1m)}^t \frac1{(t-s)^{\frac1{2\alpha}}}s^{-\frac{k+4\alpha-2}{2\alpha}}\, ds,$$
    then 
    \begin{equation}\label{lemma2_decay_estimate_3}
        \begin{aligned}
            I(k,m,t) =\frac1{t^{\frac{k+2\alpha-1}{2\alpha}}} \int_{1-\frac1m}^1 \frac1{(1-z)^{\frac1{2\alpha}}} \frac1{z^{\frac{k+4\alpha-2}{2\alpha}}}\, dz 
            \leq t^{-\frac{k+2\alpha-1}{2\alpha}} (1-\frac1m)^{-\frac{k+4\alpha-2}{2\alpha}} \int_{1-\frac1{m}}^1 (1-z)^{-\frac1{2\alpha}}\, dz &\\
            = g(m)t^{-\frac{k+2\alpha-1}{2\alpha}},&
        \end{aligned}
    \end{equation}
    where $g(m) = (1-\frac1m)^{-\frac{k+4\alpha-2}{2\alpha}} \frac1{m^{1-\frac1{2\alpha}}}$.
    Take $m = m(k) = k^{\frac{(2\alpha-1)k-(2\alpha+1)}{n+k+1}}$,
    then $g(m) \to 0$ as $k\to \infty$. 
    Thus $|g(m)|\leq C$ are uniformly bounded.
    Therefore, Lemma~\ref{lemma2_decay} follows from \eqref{lemma2_decay_estimate_1}, \eqref{lemma2_decay_estimate_2}, \eqref{lemma2_decay_estimate_3}.
\end{proof}

Let $\widetilde{GX^k} = \bigcap_{l=0}^k GX^l$ equipped with the norm $\|\cdot\|_{\widetilde{GX^k}} = \sum_{l=0}^k\|\cdot\|_{GX^l}$.
Then Lemma~\ref{lemma2_decay} implies that 
$$\|B(u,v)\|_{GX^k}
\leq C_{0}(k)\|u\|_{GX^0}\|v\|_{GX^0} + C_1\|u\|_{GX^0}\|v\|_{GX^k} + C_1\|u\|_{GX^k}\|v\|_{GX^0} 
+ C(k) \|u\|_{\widetilde{GX^{k-1}}} \|v\|_{\widetilde{GX^{k-1}}}$$
for any $k\geq 1$.
And we already know $\|B(u,v)\|_{GX^0}\leq C\|u\|_{GX^0} \|v\|_{GX^0}$.
    
Now we define an approximating sequence
$$v^{-1} = 0,\quad v^0 = S(t)u_0,$$
$$v^{j+1} = v^0 + B(v^j, v^j),$$
then for any $k\geq 0$, $v_j$ converges in $\widetilde{GX^k}$ provided that $v^0$ is small enough in $\widetilde{GX^k}$.
Moreover, we need the following lemma to conclude the proof of Theorem~\ref{decay_theorem}.

\begin{lemma}\label{lemma3_decay}
    Let $u_0$ be small enough in $\mathcal{B}$, then 
    for $\forall k \geq 0$, there exist constants $D_k, E_k$ such that 
    $$\|v^j\|_{\widetilde{GX^k}} \leq D_k,$$
    $$\|v^{j+1}-v^j\|_{\widetilde{GX^k}}\leq E_k (\frac23)^j.$$
    In particular, for any $k\geq 0$, $v^j$ converges in $\widetilde{GX^k}$.
\end{lemma}
The proof is omitted, which is exactly the same with Lemma 4.3. in \cite{ref4} 
with different formulation for $B$.
Then the lemma implies that $v^j$ to $v$ converges in $\widetilde{GX^k}$,
where $v$ is the solution in the sense of Theorem~\ref{main_theorem},
therefore this completes the proof of Theorem~\ref{decay_theorem}.

\smallskip
As a straightforward corollary of the theorem, we get the decay of space derivatives:
\begin{corollary}
    If the initial data $\|u_0\|_{\mathcal{B}}$ is small enough,
    then the solution constructed in Theorem~\ref{main_theorem} satisfies 
    $$\|\nabla^k u\|_{L^\infty} \leq C \frac1{t^{(1-\frac1{2\alpha})+\frac{k}{2\alpha}}}$$
    for any $t\geq 0$ and $k > 0$.
\end{corollary}

\section{Analyticity of the solution in space variable}
In this section, we prove the following result.
\begin{theorem}
    If $\|u_0\|_{\mathcal{B}}$ is sufficiently small, 
    then the solution constructed in Theorem~\ref{main_theorem} $u$ is analytic in space variable.
\end{theorem}
By applying the Stirling's formula, it suffices to prove 
$$\|\nabla^k u\|_{L^\infty} \lesssim C^k k^{k-1} \frac1{t^{(1-\frac1{2\alpha})+\frac{k}{2\alpha}}},$$
where $C$ is a constant independent of $k$.
And by definition, we will establish the estimate that 
$$\|u\|_{GX^k} \lesssim C^{k-1}k^{k-1},$$
where $k>K$ for some $K$ sufficiently large.

\begin{lemma}\label{lemma1_analyticity}
    $$\|\nabla^k S(t)u_0\|_{L^\infty} \leq C^{k+2}k^{\frac{k+2}{2\alpha}} t^{-(1-\frac1{2\alpha})-\frac{k}{2\alpha}}\|u_0\|_{\mathcal{B}}.$$
\end{lemma}
\begin{proof}
    Take $N$ such that $2^N\sim t^{-\frac1{2\alpha}}$, then 
    $$\begin{aligned}
        \|\nabla^k S(t)u_0\|_{L^\infty} 
    \leq& \sum_{j\leq N}\|P_j(\nabla^kS(t)u_0)\|_{L^\infty} + \sum_{j > N}\|P_j(\nabla^kS(t)u_0)\|_{L^\infty}\\
    \leq& \sum_{j\leq N}\|P_j(\nabla^kS(t)u_0)\|_{L^\infty} + C\sum_{j > N}2^{-2j}\|P_j(\nabla^{k+2}S(t)u_0)\|_{L^\infty}\\
    \leq& \sum_{j\leq N}\|(\nabla^kS(t))P_ju_0\|_{L^\infty} + C\sum_{j > N}2^{-2j}\|(\nabla^{k+2}S(t))P_ju_0\|_{L^\infty}\\
    \leq& \sum_{j\leq N}\left(C\left(\frac{k}{t}\right)^{\frac1{2\alpha}}\right)^k\|P_ju_0\|_{L^\infty} + C\sum_{j > N}2^{-2j}\left(C\left(\frac{k+2}{t}\right)^{\frac1{2\alpha}}\right)^{k+2}\|P_ju_0\|_{L^\infty}\\
    \leq& C^{k+2}\|u_0\|_{\mathcal{B}} \left( \sum_{j\leq N} \left(\frac{k}{t}\right)^{\frac{k}{2\alpha}} 2^{j(2\alpha-1)} + C\sum_{j > N}\left(\frac{k+2}{t}\right)^{\frac{k+2}{2\alpha}} 2^{-2j+j(2\alpha-1)} \right)\\
    \lesssim& C^{k+2}\|u_0\|_{\mathcal{B}} \left( \left(\frac{k}{t}\right)^{\frac{k}{2\alpha}} 2^{N(2\alpha-1)} + C\left(\frac{k+2}{t}\right)^{\frac{k+2}{2\alpha}} 2^{-2N+N(2\alpha-1)} \right)\\
    \lesssim& C^{k+2}\|u_0\|_{\mathcal{B}}k^{\frac{k+2}{2\alpha}} \left( t^{-\frac{k+2\alpha-1}{2\alpha}} + t^{-\frac{k+2-2+(2\alpha-1)}{2\alpha}} \right)\\
    \lesssim& C^{k+2}\|u_0\|_{\mathcal{B}}k^{\frac{k+2}{2\alpha}} t^{-\frac{k+2\alpha-1}{2\alpha}}
    ,\end{aligned}
    $$ 
    where the fourth inequality follows from applying \eqref{estimate_nabla_fractional_temperature} $k$ times,
    the fifth inequality follows from the definition of Besov space and Theorem~\ref{equivalence_of_norms}.
\end{proof}

\bigskip
Now we show the key result 
\begin{equation}\label{key_result}
    \|u\|_{GX^k}\leq C^{k-1}k^{k-1}.
\end{equation} 
by induction on $k$.
Recall a combinatorial result from \cite{ref5}.
\begin{lemma}\label{combinaorial_result}
    Let $\delta > \frac12$, then there exists some constant $C = C(\delta) > 0$, such that 
    $$\sum_{\gamma\leq \alpha} \left(\begin{matrix}\alpha \\ \gamma\end{matrix}\right) |\gamma|^{|\gamma|-\delta} |\alpha-\gamma|^{|\alpha-\gamma|-\delta} 
    \leq C |\alpha|^{|\alpha|-\delta}$$
    for all $\alpha\in\mathbb N^n_0$.
\end{lemma}

From Lemma~\ref{lemma2_decay}, we know 
$$\|B(u,v)\|_{GX^k}
\leq C_{0}(k)\|u\|_{GX^0}\|v\|_{GX^0} + C_1\|u\|_{GX^0}\|v\|_{GX^k} + C_1\|u\|_{GX^k}\|v\|_{GX^0} 
+ C\sum_{l=1}^{k-1} \left(\begin{matrix}k \\ l\end{matrix}\right)\|u\|_{GX^l} \|v\|_{GX^{k-l}}.$$
Now let us assume that \eqref{key_result} is valid for $1,\cdots,k-1$, then for $k$,
using the induction hypothesis and applying Lemma~\ref{combinaorial_result}, we have 
$$\sum_{l=1}^{k-1}\left(\begin{matrix}k \\ l\end{matrix}\right) \|u\|_{GX^l} \|u\|_{GX^{k-l}}
\leq \sum_{l=1}^{k-1} \left(\begin{matrix}k \\ l\end{matrix}\right) C^{l-1}k^{l-1} C^{k-l-1}(k-l)^{k-l-1}\leq C^{k-2}k^{k-1}.$$
Then apply Lemma~\ref{lemma1_analyticity}, we have 
$$\begin{aligned}
    \|u\|_{GX^k} \leq& \|S(t)u_0\|_{GX^k} + \|B(u,u)\|_{GX^k}\\
    \lesssim& C^{k-1}k^{\frac{k+2}{2\alpha}} \|u_0\|_{\mathcal{B}} + C_{0}(k)\|u\|^2_{GX^0} + 2C_1\|u\|_{GX^0}\|u\|_{GX^k} + C^{k-1}k^{k-1}      
.\end{aligned}
$$
Theorem~\ref{main_theorem} tells us $\|u\|_{GX^0}$ is small, 
so that the term $2C_1\|u\|_{GX^0}\|u\|_{GX^k}$ can be incorporated into LHS.
Note that $k^{\frac{k+2}{2\alpha}} \leq k^{k-1}$ for all sufficiently large $k$.
(For small $k$, we could choose $C$ large enough.)
Thus the theorem follows from the construction of $C_0(k)$ in \eqref{lemma2_decay_estimate_1}
and the choice of $m=m(k) = k^{\frac{(2\alpha-1)k-(2\alpha+1)}{n+k+1}}$
so that $C_0(k) = C^k k^{\frac{k+1}{2\alpha}} m^{\frac{k+n+1}{2\alpha}}$ is of the form $C^{k-1}k^{k-1}$ 
from the proof of Theorem~\ref{decay_theorem}.
Here we completes the proof.

\bibliographystyle{amsplain}
\bibliography{paper}

\providecommand{\bysame}{\leavevmode\hbox to3em{\hrulefill}\thinspace}
\providecommand{\MR}{\relax\ifhmode\unskip\space\fi MR }
\providecommand{\MRhref}[2]{%
  \href{http://www.ams.org/mathscinet-getitem?mr=#1}{#2}
}
\providecommand{\href}[2]{#2}
\begin{thebibliography}{1}

\bibitem{ref3}
Hajer Bahouri, Jean-Yves Chemin, and Rapha\"{e}l Danchin, \emph{Fourier
  analysis and nonlinear partial differential equations}, Grundlehren der
  Mathematischen Wissenschaften [Fundamental Principles of Mathematical
  Sciences], vol. 343, Springer, Heidelberg, 2011. \MR{2768550}

\bibitem{ref1}
Matteo Bonforte, Yannick Sire, and Juan~Luis V\'{a}zquez, \emph{Optimal
  existence and uniqueness theory for the fractional heat equation}, Nonlinear
  Anal. \textbf{153} (2017), 142--168. \MR{3614666}

\bibitem{ref6}
Jean Bourgain and Nata\v{s}a Pavlovi\'{c}, \emph{Ill-posedness of the
  {N}avier-{S}tokes equations in a critical space in 3{D}}, J. Funct. Anal.
  \textbf{255} (2008), no.~9, 2233--2247. \MR{2473255}

\bibitem{ref4}
Pierre Germain, Nata\v{s}a Pavlovi\'{c}, and Gigliola Staffilani,
  \emph{Regularity of solutions to the {N}avier-{S}tokes equations evolving
  from small data in {${\rm BMO}^{-1}$}}, Int. Math. Res. Not. IMRN (2007),
  no.~21, Art. ID rnm087, 35. \MR{2352218}

\bibitem{ref5}
Charles Kahane, \emph{On the spatial analyticity of solutions of the
  {N}avier-{S}tokes equations}, Arch. Rational Mech. Anal. \textbf{33} (1969),
  386--405. \MR{245989}

\bibitem{ref0}
Herbert Koch and Daniel Tataru, \emph{Well-posedness for the {N}avier-{S}tokes
  equations}, Adv. Math. \textbf{157} (2001), no.~1, 22--35. \MR{1808843}

\bibitem{ref2}
P.~G. Lemari\'{e}-Rieusset, \emph{Recent developments in the {N}avier-{S}tokes
  problem}, Chapman \& Hall/CRC Research Notes in Mathematics, vol. 431,
  Chapman \& Hall/CRC, Boca Raton, FL, 2002. \MR{1938147}

\end{thebibliography}

\end{document}